\documentclass[reqno,12pt] {amsart}
\usepackage{amsmath, mathrsfs}

\newtheorem{theorem}{Theorem}[section]
\newtheorem{lemma}[theorem]{Lemma}

\theoremstyle{definition}
\newtheorem{definition}[theorem]{Definition}

\theoremstyle{remark}
\newtheorem{remark}[theorem]{Remark}

\newcommand{\mysection}[1]{\section{#1}
\setcounter{equation}{0}}

\newcommand{\bR}{\mathbb R}

\newcommand\cP{\mathscr{P}}

\newcommand{\Div}{\operatorname{div}}

\makeatletter
\def\dashint{\operatorname%
{\,\,\text{\bf--}\kern-.98em\DOTSI\intop\ilimits@\!\!}}
\makeatother

\renewcommand{\epsilon}{\varepsilon}

\begin{document}
\title [The aggregation equation with power-law kernels]{The aggregation equation with power-law kernels: ill-posedness, mass concentration and similarity solutions}

\author[H. Dong]{Hongjie Dong}
\address[H. Dong]{Division of Applied Mathematics, Brown University, 182 George Street, Box F, Providence, RI 02912, USA}
\email{Hongjie\_Dong@brown.edu}
\thanks{Hongjie Dong was partially supported by the National Science Foundation under agreement No. DMS-0800129.}

\date{\today}



\begin{abstract}
We study the multidimensional aggregation equation $u_t+\Div(uv)=0$, $v=-\nabla K*u$ with initial data in $\cP_2(\bR^d)\cap L_{p}(\bR^d)$. We prove that with biological relevant potential $K(x)=|x|$, the equation is ill-posed in the critical Lebesgue space $L_{d/(d-1)}(\bR^d)$ in the sense that there exists initial data in $\cP_2(\bR^d)\cap L_{d/(d-1)}(\bR^d)$ such that the unique measure-valued solution leaves $L_{d/(d-1)}(\bR^d)$ immediately. We also extend this result to more general power-law kernels $K(x)=|x|^\alpha$, $0<\alpha<2$ for $p=p_s:=d/(d+\alpha-2)$, and prove a conjecture in \cite{BLR} about instantaneous mass concentration for initial data in $\cP_2(\bR^d)\cap L_{p}(\bR^d)$ with $p<p_s$. Finally, we characterize all the ``first kind'' radially symmetric similarity solutions in dimension greater than two.
\end{abstract}

\maketitle

\mysection{Introduction}

In this paper we consider the multidimensional aggregation equation
\begin{equation}
                            \label{Agg}
u_t+\Div(uv)=0,\quad v=-\nabla K*u
\end{equation}
for $x\in\bR^d$ and $t> 0$ with the initial data
\begin{equation*}
u(0,x)=u_0(x),\quad x\in\bR^d.
\end{equation*}
Here $d\ge 2$, $u\ge 0$, $K$ is the interaction potential, and $*$ denotes the spatial convolution. This equation arises in various models for biological aggregation and problems in granular media; see, for instance, \cite{ME,CMV,LT}. The problems of the well-posedness in different spaces, finite-time blowups, asymptotic behaviors of solutions of this equation, as well as the equation with an additional dissipation term, have been studied extensively by a number of authors; see \cite{LT, BV,BL,La,BCL,LR1,LR2,BLR,CFFLS} and reference therein. We refer the reader to \cite{BL2} for a nice review about recent progress on the aggregation equation.

In \cite{BLR} Bertozzi, Laurent and Rosado studied comprehensively the $L_p$ theory for the aggregation equation \eqref{Agg}. Among some other results, they considered radially symmetric kernels where the singularity at the origin is of order $|x|^\alpha$ for some $\alpha>2-d$, and proved the local well-posedness of \eqref{Agg} in $\cP_2(\bR^d)\cap L_p(\bR^d)$ for any $p>p_s$, where $p_s=d/(d+\alpha-2)$ (see below for the definition of the space $\cP_2(\bR^d)$). In the biological relevant case $K(x)=|x|$, they showed that solutions can concentrate mass instantaneously for initial data in $\cP_2(\bR^d)\cap L_p(\bR^d)$ for any $p<p_s$.
It remains unknown if \eqref{Agg} is well-posed in the critical space $\cP_2(\bR^d)\cap L_{p_s}(\bR^d)$. Another interesting open question is whether one can show a similar instantaneous mass concentration phenomenon for the equation with general power-law potential $|x|^\alpha$. The authors conjectured in \cite{BLR} that the answer to the second question is positive.

The aim of the current paper is to answer these questions. For the first question, we shall construct radially symmetric initial data in $\cP_2(\bR^d)\cap L_{d/(d-1)}(\bR^d)$, such that the unique measured-valued solution leaves $L_{d/(d-1)}(\bR^d)$ immediately for $t>0$; see Theorem \ref{thm1}. This result implies that \eqref{Agg} is ill-posed in $\cP_2(\bR^d)\cap L_{d/(d-1)}(\bR^d)$, and the well-posedness result for $p>d/(d-1)$ obtained in \cite{BLR} is sharp. For the second question, we show that, for any $\alpha\in (0,2)$ and any $p<p_s$, there exists radially symmetric initial data in $\cP_2(\bR^d)\cap L_p(\bR^d)$ such that the solution concentrates mass at the origin instantaneously. In other words, a Dirac delta  appears immediately in the solution. Therefore, we settle down the aforementioned conjecture in \cite{BLR}. We also prove that, for any $\alpha\in (0,2)$, \eqref{Agg} is ill-posed in $\cP_2(\bR^d)\cap L_{p_s}(\bR^d)$ by constructing initial data in $\cP_2(\bR^d)\cap L_{p_s}(\bR^d)$ such that any weakly continuous measured-valued solution, if exists, leaves $L_{p_s}(\bR^d)$ immediately for $t>0$.

The proofs use some ideas in \cite{BLR} by considering the flow map driven by the velocity field $v$. Roughly speaking, there are two steps in the proofs. In the first step, we find a suitable representation of the velocity field $v$ in the polar coordinates, and prove the monotonicity, positivity and asymptotics of the corresponding kernel. For $K(x)=|x|$, these have already been established in \cite{BLR} (Lemma \ref{lem2.3}). More delicate analysis is needed for general power-law potential $K(x)=|x|^\alpha$ (see Lemma \ref{lem3.3}). In the second step, we deduce certain positive lower bounds for the velocity (Lemmas \ref{lem2.5}, \ref{lem3.5} and \ref{lem4.5}). Combined with the monotonicity of the velocity in time, we then reduce the problems to study the dynamics of solutions to some ordinary differential equations. In the case $p<p_s$, it is shown that the flow map reaches the origin in a short time, which generates a Dirac delta. While in the critical case $p=p_s$, a Dirac delta may not develop shortly, but the flow map makes the mass concentrate quickly enough near the origin such that the solution $u$ leaves $L_{p_s}(\bR^d)$ immediately.

We also consider profiles of similarity solutions to \eqref{Agg} at the blowup time with the potential $K(x)=|x|$, which conserve mass. This type of solutions is an example of  ``first-kind'' similarity solutions; see \cite{Ba}. In \cite{BCL}, Bertozzi, Carrillo and Laurent constructed  radially symmetric first-kind similarity solutions in the dimension one and two, and proved that in any odd dimension $d\ge 3$  such solutions cannot exist with support on open sets. By observing certain concavity property of the kernel in the polar coordinates, in Section \ref{sec3} we characterize all the radially symmetric first-kind similarity solutions in the dimension $d\ge 3$.

We finish the Introduction by fixing some notation. Most notation in this paper are chosen to be compatible with those in \cite{BLR}. For $r>0$, let
$$
B_r=\{x\in \bR^d:|x|<r\},\quad S_r=\{x\in \bR^d:|x|=r\}.
$$
By $\omega_d$ we mean the surface area of the unit sphere $S_1$ in $\bR^d$. We denote $\cP(\bR^d)$ to be the set of all probability measures on $\bR^d$, and $\cP_2(\bR^d)$ to be the set of all probability measures on $\bR^d$ with bounded second moment:
$$
\cP_2(\bR^d):=\left\{\mu\in \cP(\bR^d):\int_{\bR^d}|x|^2\,d\mu(x)<\infty\right\}.
$$

\mysection{Ill-posedness when $K(x)=|x|$}
                    \label{sec2}

In this section, we prove the following result, which reads that with potential function $K(x)=|x|$, the aggregation equation \eqref{Agg} is ill-posed in $\cP_2\cap L_{d/(d-1)}(\bR^d)$.

\begin{theorem}
                            \label{thm1}
Let $K(x)=|x|$, $k\in ((d-1)/d, 1)$ and
\begin{equation}
                                    \label{eq22.15}
u_0(x)=\frac L {|x|^{d-1}(-\log |x|)^{k}}1_{|x|\le 1/2}\in \cP_2(\bR^d)\cap L_{d/(d-1)}(\bR^d),
\end{equation}
where
$$
L=\int_{|x|\le 1/2}|x|^{-d+1}(-\log |x|)^{-k}\,dx
$$
is a normalization constant. Let $(\mu_t)_{t\in (0,\infty)}$ be the unique measure-valued solution to the aggregation equation \eqref{Agg}. Then for any $t>0$ the density of $\mu_t$, if exists, is not in $L_{d/(d-1)}(\bR^d)$.
\end{theorem}

We note that for $K(x)=|x|^\alpha,\alpha\ge 1$ the global existence and uniqueness of a weakly continuous measure-valued solution was proved in \cite{CFFLS}. Moreover, for $\alpha<2$, any measure-valued  solution will eventually collapse to a Dirac delta at the center of mass in a finite time. In \cite{BLR}, it was also proved that if $u_0\in \cP_2$ then the measure-valued solution stays in $\cP_2$.

\begin{definition}
Let $\mu\in \cP(\bR^d)$ be a radially symmetric probability measure. We define $\hat \mu\in \cP([0,+\infty))$ by
$$
\hat \mu(I)=\mu(\{x\in \bR^d: |x|\in I\})
$$
for all Borel sets $I$ in $[0,+\infty))$.
\end{definition}

We reformulate some results in \cite[Sect. 4]{BLR} as the following lemmas.

\begin{lemma}
                        \label{lem2.3}
Let $K(x)=|x|$, $\mu\in \cP(\bR^d)$ be a radially symmetric measure. Then for any $x\neq 0$, we have
$$
(\nabla K*\mu)(x)=\int_0^\infty \psi(\rho/|x|)\,d\hat \mu(\rho) \frac x {|x|},
$$
where $\psi:[0,\infty)\to \bR$ is a function  defined by
$$
\psi(\rho)=\dashint_{S_1}\frac {1-\rho y_1} {|e_1-\rho y|}\,d\sigma(y).
$$
Moreover, $\psi$ is continuous, positive, non-increasing on $[0,\infty)$, and
$$
\psi(0)=1,\quad \lim_{\rho\to \infty}\psi(\rho)\rho=\frac {d-1} d.
$$
\end{lemma}

\begin{lemma}
                        \label{lem2.4}
Let $K(x)=|x|$, and $(\mu_t)_{t\in [0,\infty)}$ be a radially symmetric weakly continuous measured-valued solution of the aggregation equation \eqref{Agg}. Then the vector field
$$
v(t,x)=-(\nabla K*\mu_t)1_{x\neq 0}
$$
is continuous on $[0,\infty)\times (\bR^d\setminus \{0\})$. For any $t\ge 0$, we have
$$
\mu_t=X_t^\#\mu_0,
$$
where, for each $x\in \bR^d$, $X_t=X_t(x)$ is an absolutely continuous function on $[0,\infty)$ satisfying
\begin{equation*}
\left\{
  \begin{aligned}
    \frac d {dt} X_t(x)&= v(t,X_t(x)) \quad\hbox{for a.e. $t\in (0,\infty)$;} \\
    X_0(x)&=x,
  \end{aligned}
\right.
\end{equation*}
and $X_t^\#$ means the push-forward of a measure by the map $X_t$.
Moreover, for any $x\neq 0$, $X_t(x)=R_t(|x|)x/|x|$, where $R_t$ is an absolutely continuous, non-negative and non-increasing function in $t\in [0,\infty)$, and
$$
\hat \mu_t=R_t^\#\hat \mu_0.
$$
Consequently, for any $x\neq 0$,
$$
(\nabla K*\mu_t)(x)=\int_0^\infty \psi(R_t(\rho)/|x|)\,d\hat \mu_0(\rho) \frac x {|x|},
$$
and $|(\nabla K*\mu_t)(x)|$ is non-decreasing in $t$.
\end{lemma}

Next we establish a point-wise lower bound of the velocity $v$ at $t=0$.

\begin{lemma}
                                \label{lem2.5}
Let $K(x)=|x|$ and $u_0$ be the initial data defined in \eqref{eq22.15} with $k\in ((d-1)/d,1)$:
$$
u_0(x)=\frac L {|x|^{d-1}(-\log |x|)^{k}}1_{|x|\le 1/2}.
$$
Then there exists a constant $\delta_1>0$ such that
\begin{equation}
                                    \label{eq22.22}
|(\nabla K*u_0)(x)|\ge \delta_1 |x|(-\log |x|)^{1-k}
\end{equation}
for any $x\in \bR^d$ satisfying $0<|x|<1/2$.
\end{lemma}
\begin{proof}
Clearly, we have
$$
\hat u_0(\rho)=L\omega_d(-\log \rho)^{-k}1_{\rho\le 1/2}.
$$
By the positivity and continuity  of $|v(0,\cdot)|$ in $\bR^d\setminus \{0\}$, it suffices to prove \eqref{eq22.22} for $x\in \bR^d$ satisfying $0<|x|<1/4$. It follows from Lemma \ref{lem2.3} that for any $\rho\in (|x|,1/2)$,
$$
\psi(\rho/|x|)\ge \delta_0 |x|/\rho
$$
for a constant $\delta_0>0$ independent of $|x|$.
Thus, for any $x\in \bR^d$ satisfying $0<|x|<1/4$,
\begin{align*}
|(\nabla K*u_0)(x)|&\ge \int_{|x|}^{1/2}\psi(\rho/|x|)\hat u_0(\rho)\,d\rho\\
&\ge \delta_0L\omega_d|x|\int_{|x|}^{1/2}(-\log \rho)^{-k}\,\frac {d\rho}{\rho}\\
&=\frac {\delta_0L\omega_d} {1-k}|x|\left((-\log |x|)^{1-k}-(\log 2)^{1-k}\right)\\
&\ge \delta_1 |x|(-\log |x|)^{1-k},
\end{align*}
since $0<|x|<1/4$. The lower bound \eqref{eq22.22} is proved.
\end{proof}

We are now ready to prove Theorem \ref{thm1}.
\begin{proof}[Proof of Theorem \ref{thm1}]
By Lemma \ref{lem2.4}, for each $x\in \bR^d$, $|v(t,x)|$ is non-decreasing in $t$. Therefore, from the lower bound of $|v(x,0)|$ \eqref{eq22.22} we infer that for any $r\in (0,1/2)$,
\begin{equation}
                            \label{eq23.09}
\frac d {dt} R_t(r)\le -\delta_1 R_t(r)(-\log R_t(r))^{1-k}.
\end{equation}
Solving the ordinary differential inequality \eqref{eq23.09} gives
\begin{equation}
                                    \label{eq23.18}
(-\log R_t(r))^k\ge (-\log r)^k+k\delta_1 t.
\end{equation}
From \eqref{eq23.18}, we obtain
\begin{equation*}
\frac {(-\log R_t(r))^k} {(-\log r)^k}\ge 1 +\frac {k\delta_1 t} {(-\log r)^k},
\end{equation*}
and by Taylor's formula, for $t$ sufficiently small,
$$
\frac {(-\log R_t(r))} {(-\log r)}\ge 1 +\frac {\delta_2 t} {(-\log r)^k},
$$
where $\delta_2=\delta_1/2$.
We thus obtain
\begin{equation}
                                        \label{eq23.27}
R_t(r)\le r e^{-\delta_2 t (-\log r)^{1-k}}.
\end{equation}
Now we suppose that, for some $t>0$, $\mu_t$ has a density function $u(t,x)\in L_{d/(d-1)}(\bR^d)$. By H\"older's inequality, for any $r\in (0,1/2)$
\begin{align}
\mu_t(B_{R_t(r)})&=\int_{B_{R_t(r)}}u(t,x)\,dx\nonumber\\
&\le \left(\int_{B_{R_t(r)}}u^{\frac d {d-1}}(t,x)\,dx\right)^{\frac {d-1}d }|B_{R_t(r)}|^{\frac 1 d }\nonumber\\
&\le \|u(t,\cdot)\|_{L_{d/(d-1)}(\bR^d)}R_t(r).
                                    \label{eq23.42}
\end{align}
On the other hand, by the definitions of $\mu_t$ and $R_t$,
\begin{equation}
                                \label{eq23.49}
\mu_t(B_{R_t(r)})\ge \mu_0(B_r)=L\omega_d\int_0^r (-\log\rho)^{-k}\,d\rho.
\end{equation}
Note that the above inequality is strict only when there is a mass concentration before time $t$.
We combine \eqref{eq23.42}, \eqref{eq23.49} and \eqref{eq23.27} to get
\begin{align}
\|u(t,\cdot)\|_{L_{d/(d-1)}(\bR^d)}&\ge \frac {\mu_t(B_{R_t(r)})} {R_t(r)}\nonumber\\
&\ge L\omega_d\frac {\int_0^r (-\log\rho)^{-k}\,d\rho}{r e^{-\delta_2 t (-\log r)^{1-k}}}.
                                    \label{eq00.57}
\end{align}
However, by L'Hospital's rule,
\begin{align*}
&\lim_{r\searrow 0}\frac {\int_0^r (-\log\rho)^{-k}\,d\rho}{r e^{-\delta_2 t (-\log r)^{1-k}}}\\
&\,=\lim_{r\searrow 0}\frac {(-\log r)^{-k}}{e^{-\delta_2 t (-\log r)^{1-k}}\big(1+\delta_2 t(1-k)(-\log r)^{-k}\big)}\\
&\,=\infty.
\end{align*}
This gives a contradiction to \eqref{eq00.57} since we assume $u(t,\cdot)\in L_{d/(d-1)}(\bR^d)$. The theorem is proved.
\end{proof}

\mysection{Similarity solutions}
                        \label{sec3}
In this section, we consider the problem of similarity solutions to the aggregation equation.
This problem is closely related to the blowup profile for \eqref{Agg} at the blowup time.
Let us consider mass-conserving similarity solutions
\begin{equation}
                                    \label{eq00.25}
u(t,x)=\frac 1 {R(t)^d} u_0\left(\frac x {R(t)}\right)
\end{equation}
to the aggregation equation \eqref{Agg} with interaction kernel $K(x)=|x|$. These solutions are ``first-kind'' similarity solutions, while ``second-kind'' similarity solutions do not conserve mass. 

If $u$ is a radially symmetric first-kind similarity solution given by \eqref{eq00.25}, then by the homogeneity it is easily seen that
$$
v(t,x)=v_0\left(\frac x {R(t)}\right),\quad v_0=-\nabla |x|*u_0.
$$
Moreover, it was proved in  \cite{BCL} that $R(t)$ must be a linear function and on the support of $u_0$
\begin{equation}
                                    \label{eq16.08}
v_0=-\lambda x
\end{equation}
for some constant $\lambda>0$. In one space dimension, the authors of \cite{BCL} constructed a first-kind similarity solution
$$
u(t,x)=\frac 1 {T^*-t} U\left(\frac x {T^*-t}\right),
$$
where $U$ is the uniform distribution on $[-1,1]$. In any dimension, there is a first-kind similarity measure-valued solution which is a single delta on a sphere with radius shrinking linearly in time \cite[Remark 3.8]{BCL}. Such solution is called single delta-ring solution. In two space dimension, a two delta-ring solution was constructed in the same paper, i.e., $\hat\mu_0=m_1\delta_{\rho_1}+m_2\delta_{\rho_2}$ for some $0<\rho_1<\rho_2<\infty$ and $m_1,m_2>0$. On the other hand, for any odd dimension $d\ge 3$, by using a relation between $K(x)$ and the Newtonian potential they proved the non-existence of radially symmetric first-kind similarity solutions with support on open sets.  

We characterize all the radially symmetric first-kind similarity measure-valued solutions in the following theorem, which in particular implies that in dimension three and higher there cannot exist first-kind similarity solutions with support on open sets or multi delta rings.

\begin{theorem}[Characterization of similarity solutions]
                                    \label{thm3.1}
Let $d\ge 3$ and $K(x)=|x|$. Then any radially symmetric first-kind similarity measure-valued solution is of the form
$$
\mu_t(x)=\frac 1 {R(t)^d} \mu_0\left(\frac x {R(t)}\right),
$$
where
\begin{equation}
                                    \label{eq10.40}
\hat \mu_0=m_0\delta_0+m_1\delta_{\rho_1}
\end{equation}
for some constants $m_0,m_1\ge 0$ and $\rho_1>0$.
\end{theorem}

For the proof, first we recall that by Lemma \ref{lem2.3} for any radially symmetric measure-valued solution $\mu_t$ and $t>0$,
\begin{equation}
                                            \label{eq16.39}
v(t,x)=-\int_0^\infty \phi(|x|/\rho)\,d\hat\mu_t(\rho)\frac {x} {|x|},                             \end{equation}
where
$\phi:[0,\infty)\to \bR$ is a function  defined by
\begin{equation}
                            \label{eq16.38}
\phi(r)=\dashint_{S_1}\frac {r-y_1} {|re_1-y|}\,d\sigma(y).
\end{equation}

The proof of Theorem \ref{thm3.1} relies on the following observation.

\begin{lemma}
                                \label{lem3.2}
i) Let $d\ge 4$ and $K(x)=|x|$. Then the function $\phi$ defined by \eqref{eq16.38} is $C^2$ on $[0,\infty)$ and satisfies
\begin{equation}
                                        \label{eq10.35}
\phi(0)=0,\quad \lim_{r\to \infty}\phi(r)=1,\quad \phi'(r)>0\,\,\text{on}\,\,[0,\infty),
\end{equation}
$$
\phi''(r)<0\,\,\text{on}\,\,(0,\infty).
$$

ii) Let $d=3$ and $K(x)=|x|$. Then the function $\phi$ defined by \eqref{eq16.38} is $C^1$ on $[0,\infty)$ and satisfies \eqref{eq10.35}. Moreover, it is concave on $[0,\infty)$, linear on $[0,1]$ and strictly concave on $(1,\infty)$.
\end{lemma}

Suppose for the moment that Lemma \ref{lem3.2} is verified. We prove Theorem \ref{thm3.1} by a contradiction argument. Suppose $(\mu_t)$ is a radially symmetric first-kind similarity measure-valued solution such that there are two numbers
$$0<\rho_1<\rho_2,\quad \rho_1,\rho_2\in\text{supp}\,\hat\mu_0.
$$
Denote $w(\rho):=|v_0(\rho e_1)|$. By \eqref{eq16.39},
$$
w(\rho_k)=\int_0^\infty \phi(\rho_k/\rho)\,d\hat \mu_0(\rho),\quad k=1,2.
$$
It follows from Lemma \ref{lem3.2} that
$$
\phi(\rho_1/\rho)\ge \frac {\rho_1}{\rho_2} \phi(\rho_2/\rho),
$$
and the inequality is strict for all $\rho$ in a small neighborhood of $\rho_1$ because $\rho_2/\rho_1>1$. Since $\rho_1\in \text{supp}\,\hat\mu_0$, we get
$$
w(\rho_1)> \frac {\rho_1}{\rho_2} w(\rho_2).
$$
On the other hand, from \eqref{eq16.08} $w$ is a linear function on $\text{supp}\,\hat\mu_0$:
$$
w(\rho_1)= \frac {\rho_1}{\rho_2} w(\rho_2).
$$
Therefore, we reach a contradiction.

We finish this section by proving Lemma \ref{lem3.2}.

\begin{proof}[Proof of Lemma \ref{lem3.2}]
We rewrite \eqref{eq16.38} as
\begin{equation}
                                    \label{eq16.59}
\phi(r)=\frac {\omega_{d-1}}{\omega_d}\int_0^\pi \frac {(r-\cos\theta)(\sin\theta)^{d-2}}{A(r,\theta)}\,d\theta,
\end{equation}
where
$$
A(r,\theta)=(1+r^2-2r\cos\theta)^{1/2}.
$$
For $r\in [0,1)\cap (1,\infty)$, a direct computation gives
\begin{align}
                                \label{eq17.02}
\phi'(r)&=\frac {\omega_{d-1}}{\omega_d}\int_0^\pi (\sin\theta)^{d-2}\left(\frac 1 A-\frac {(r-\cos\theta)^2}{A^3}\right)\,d\theta\nonumber\\
&=\frac {\omega_{d-1}}{\omega_d}\int_0^\pi \frac {(\sin\theta)^{d}} {A^3}\,d\theta,
\end{align}
and
\begin{equation}
                                \label{eq17.06}
\phi''(r)=-\frac {3\omega_{d-1}}{\omega_d}\int_0^\pi \frac {(\sin\theta)^d(r-\cos\theta)} {A^5}\,d\theta.
\end{equation}
Integration by parts yields
\begin{align*}
\phi''(r)&=-\frac {3\omega_{d-1}}{\omega_d}\int_0^\pi \frac {r(\sin\theta)^d} {A^5}-
\frac {\left((\sin\theta)^{d+1}\right)'} {(d+1)A^5}\,d\theta\\
&=-\frac {3\omega_{d-1}}{\omega_d}\int_0^\pi \frac {r(\sin\theta)^d} {A^5}-
\frac {5(\sin\theta)^{d+1}r\sin\theta} {(d+1)A^7}\,d\theta\\
&=-\frac {3\omega_{d-1}}{\omega_d}\int_0^\pi \frac {r(\sin\theta)^d} {A^7}
\left(r^2+1-2r\cos\theta-\frac 5 {d+1} \sin^2\theta\right)\,d\theta.
\end{align*}

{\em Case 1: $d\ge 4$.} In this case, we have
$$
r^2+1-2r\cos\theta-\frac 5 {d+1} \sin^2\theta\ge (r-\cos\theta)^2.
$$
Therefore, we get $\phi''<0$ on $(0,1)\cap (1,\infty)$.
Note that because
$$
|r-\cos\theta|\le A,\quad r |\sin\theta|\le A,
$$
the integrals on the right-hand sides of \eqref{eq17.02} and \eqref{eq17.06} are absolutely convergent and continuous at $r=1$ by the dominated convergence theorem. Thus \eqref{eq17.02} and \eqref{eq17.06} hold on the whole region $[0,\infty)$. So we conclude $\phi\in C^2([0,\infty))$ and $\phi''<0$ on $(0,\infty)$. Moreover, since $\phi(0)=0$, $\phi\to 1$ as $r\to \infty$ and $\phi'>0$ by \eqref{eq17.02}, $\phi$ is bounded and non-negative on $[0,\infty)$.

{\em Case 2: $d=3$.} As before, the integral on the right-hand side of \eqref{eq17.02} is absolutely convergent and continuous at $r=1$ by the dominated convergence theorem. Thus \eqref{eq17.02} holds on the whole region $[0,\infty)$, and we conclude $\phi\in C^1([0,\infty))$ and $\phi'>0$ on $[0,\infty)$. For $d=3$, one can explicitly compute $\phi$. From \eqref{eq16.59}, we have
\begin{align}
                                    \label{eq11.22}
\phi(r)&=\frac 1 2\int_0^\pi \frac {(r-\cos\theta)\sin\theta}{A(r,\theta)}\,d\theta\nonumber\\
&=\frac 1 2\int_0^\pi \partial_\theta A-\frac {\cos\theta\sin\theta}{A}\,d\theta\nonumber\\
&=\frac 1 2(|1+r|-|1-r|)-I,
\end{align}
where,
\begin{align*}
I:=\frac 1 2\int_0^\pi \frac {\cos\theta\sin\theta}{A}\,d\theta.
\end{align*}
Integrating by parts and using \eqref{eq17.02}, we get
$$
I=\frac 1 4\int_0^\pi \frac {(\sin^2\theta)'}{A}\,d\theta
=\frac 1 4\int_0^\pi \frac {r\sin^3\theta}{A^3}\,d\theta=\frac r 2 \phi'(r).
$$
Thus, $\phi$ satisfies
$$
\phi(r)=\left\{
  \begin{aligned}
    &r- \frac {r} 2 \phi'(r)\quad\hbox{for $r\in [0,1]$;} \\
    &1- \frac r 2 \phi'(r)\quad\hbox{for $r\in (1,\infty)$,}
  \end{aligned}
\right.
$$
and $\phi(1)=2/3$ by \eqref{eq16.59}. Solving this ordinary differential equation, we obtain
$$
\phi(r)=\left\{
  \begin{aligned}
    &{2r}/3\quad\hbox{for $r\in [0,1]$;} \\
    &1- \frac {1} {3r^2}\quad\hbox{for $r\in (1,\infty)$,}
  \end{aligned}
\right.
$$
which immediately shows all the remaining claims in ii). The lemma is proved.
\end{proof}

\begin{remark}
In the case $d=2$, it is easy to check that $\phi$ is convex in $(0,1)$, concave in $(1,\infty)$, and $\phi'\to+\infty$ as $r\to 1$.
\end{remark}



\mysection{Ill-posedness and instantaneous concentration when $K(x)=|x|^\alpha$}

In this section, we study the ill-posedness and the instantaneous mass concentration phenomenon for general power-law kernel $K(x)=|x|^\alpha,0<\alpha<2$. In \cite{BLR}, the authors established the instantaneous mass concentration for certain initial data in $L_p,p<p_s$ in the special case $\alpha=1$, and conjectured that it remains true for general $\alpha$. We prove this conjecture in this section.

\begin{theorem}[Instantaneous mass concentration]
                            \label{thm2}
Let $K(x)=|x|^\alpha,\alpha\in (0,2)$, $p_s=d/(d+\alpha-2)$ and
\begin{equation}
                                \label{eq22.16}
u_0(x)=\frac L {|x|^{d+\alpha-2+\epsilon}}1_{|x|\le 1/2},
\end{equation}
where
$$
\epsilon\in (0,1),\quad
L=\int_{|x|\le 1/2}|x|^{-d-\alpha+2-\epsilon}\,dx
$$
is a normalization constant. Note that $u_0\in \cP_2(\bR^d)\cap L_{p}(\bR^d)$ for any $p\in [1,d/(d+\alpha-2+\epsilon))$. Let $(\mu_t)_{t\in (0,\infty)}$ be a radially symmetric weakly continuous measure-valued solution to the aggregation equation \eqref{Agg}. Then for any $t>0$ we have
$$
\mu_t(\{0\})>0.
$$
Therefore mass is concentrated at the origin instantaneously and the solution is singular with respect to the Lebesgue measure.
\end{theorem}

We shall also prove the following theorem which is a generalization of Theorem \ref{thm1}. It reads that the aggregation equation \eqref{Agg} is ill-posed in the critical space $\cP_2(\bR^d)\cap L_{p_s}(\bR^d)$.

\begin{theorem}
                            \label{thm3}
Let $K(x)=|x|^\alpha,\alpha\in (0,2)$, $p_s=d/(d+\alpha-2)$, $k\in (1/p_s, 1)$ and
\begin{equation}
                                            \label{eq22.17}
u_0(x)=\frac L {|x|^{d+\alpha-2}(-\log |x|)^{k}}1_{|x|\le 1/2}\in \cP_2(\bR^d)\cap L_{p_s}(\bR^d),
\end{equation}
where
$$
L=\int_{|x|\le 1/2}|x|^{-d-\alpha+2}(-\log |x|)^{-k}\,dx
$$
is a normalization constant. Let $(\mu_t)_{t\in (0,\infty)}$ be a radially symmetric weakly continuous measure-valued solution to the aggregation equation \eqref{Agg}. Then for any $t>0$ the density of $\mu_t$, if exists, is not in $L_{p_s}(\bR^d)$.
\end{theorem}

\begin{remark}
In the case that $\alpha\in (2-d,0),d\ge 3$, we infer from Lemma \ref{lem3.3} below that the velocity field $v$ is repulsive in the sense that, for radially symmetric solution, $v(t,x)$ is in the $x$ direction. So one cannot expect the instantaneous mass concentration phenomenon in this case. In fact, we believe that in the case $\alpha\in (2-d,0)$, the aggregation equation \eqref{Agg} is well-posed in $L_p$ for $p\in (1,\infty)$. Formally, integrating by parts, we obtain for any $p\in (1,\infty)$,
\begin{equation*}
\frac d {dt} \int_{\bR^d}u(t,x)^p\,dx=-(p-1)\int_{\bR^d}u(t,x)^p\Div v(t,x)\,dx,\quad t>0.
\end{equation*}
By the definition,
$$
\Div v=-\Delta K*u=-{\alpha(\alpha+d-2)}|x|^{\alpha-2}*u\ge 0.
$$
Therefore, $\int_{\bR^d}u(t,x)^p\,dx$ is non-increasing in $t$.  On the other hand, if instead $v$ is defined by
$$
v=\nabla K*u,
$$
then the velocity field is attractive, and we can see from the proofs below that the results of Theorems \ref{thm2} and \ref{thm3} remain valid in the case $\alpha\in (2-d,0)$.
\end{remark}

\subsection{Proof of Theorem \ref{thm2}}

For a fixed $\alpha\in (2-d,2)$, we define a function $\psi:[0,\infty)\to \bR$ by
\begin{equation}
                                \label{eq10.37}
\psi(\rho)=\dashint_{S_1}\frac {1-\rho y_1} {|e_1-\rho y|^{2-\alpha}}\,d\sigma(y).
\end{equation}

The following lemma plays a key role in our proof.

\begin{lemma}
                        \label{lem3.3}
Let $K(x)=|x|^\alpha,\alpha\in (2-d,2)$, $\mu\in \cP(\bR^d)$ is a radially symmetric probability measure. Then for any $x\neq 0$, we have
\begin{equation}
                        \label{eq10.05}
(\nabla K*\mu)(x)=\alpha|x|^{\alpha-1}\int_0^\infty \psi(\rho/|x|)\,d\hat \mu(\rho) \frac x {|x|}.
\end{equation}
Moreover, $\psi$ is continuous, positive, non-increasing on $[0,\infty)$, and
$$
\psi(0)=1,\quad \lim_{\rho\to \infty}\psi(\rho)\rho^{2-\alpha}=\frac {d+\alpha-2} d.
$$
\end{lemma}
\begin{proof}
Equality \eqref{eq10.05} follows from a direct computation. We only prove the second part of the lemma. Since $\alpha>2-d$ and
$$
\dashint_{S_1}\frac {|1-\rho y_1|} {|e_1-\rho y|^{2-\alpha}}\,d\sigma(y)
\le \dashint_{S_1}\frac 1 {|e_1-\rho y|^{1-\alpha}}\,d\sigma(y),
$$
the integral on the right-hand side of \eqref{eq10.37} is convergent for any $\rho\in [0,\infty)$, and by the dominated convergence theorem, it is easy to see that $\psi$ is continuous on $[0,\infty)$ and $\psi(0)=1$.
We rewrite \eqref{eq10.37} as
\begin{equation}
                            \label{eq11.54}
\psi(\rho)=\frac {\omega_{d-1}}{\omega_d}\int_0^\pi \frac {(1-\rho\cos\theta)(\sin\theta)^{d-2}}{A(\rho,\theta)^{2-\alpha}}\,d\theta,
\end{equation}
where
$$
A(\rho,\theta)=(1+\rho^2-2\rho\cos\theta)^{1/2}.
$$
For $\rho\in [0,1)\cap (1,\infty)$, a direct computation yields
\begin{equation}
                                        \label{eq11.00}
\psi'(\rho)=\frac {\omega_{d-1}}{\omega_d}(I_1+I_2),
\end{equation}
where
\begin{align*}
I_1&:=\int_0^\pi \frac {(-\cos\theta)(\sin\theta)^{d-2}}{A(\rho,\theta)^{2-\alpha}}\,d\theta,\\
I_2&:=(\alpha-2)\int_0^\pi\frac{(1-\rho\cos\theta)(\sin\theta)^{d-2}(\rho-\cos\theta)}
{A(\rho,\theta)^{4-\alpha}}\,d\theta.
\end{align*}
Integrating by parts gives us
\begin{equation}
                                        \label{eq11.06}
I_1=-\frac 1 {d-1} \int_0^\pi \frac {\left((\sin\theta)^{d-1}\right)'}{A(\rho,\theta)^{2-\alpha}}\,d\theta
=\frac {\alpha-2} {d-1} \int_0^\pi \frac {(\sin\theta)^{d}\rho}{A(\rho,\theta)^{4-\alpha}}\,d\theta.
\end{equation}
From \eqref{eq11.00} and \eqref{eq11.06}, we get
\begin{align}
                                        \label{eq11.09}
\psi'(\rho)&=\frac {\omega_{d-1}(\alpha-2)}{\omega_d}
\int_0^\pi\frac{(\sin\theta)^{d-2}}
{A(\rho,\theta)^{4-\alpha}}\left(\frac {\rho(\sin\theta)^2}{d-1}+(1-\rho\cos\theta)(\rho-\cos\theta)\right)\,d\theta,\nonumber\\
&=\frac {\omega_{d-1}(\alpha-2)}{\omega_d}
\int_0^\pi\frac{(\sin\theta)^{d-2}}
{A(\rho,\theta)^{4-\alpha}}\left(\frac {\rho(\sin\theta)^2 d}{d-1}-
A(\rho,\theta)^2\cos\theta \right)\,d\theta\nonumber\\
&=\frac {\omega_{d-1}(\alpha-2)}{\omega_d}
\left(\int_0^\pi\frac{\rho(\sin\theta)^{d}}
{A(\rho,\theta)^{4-\alpha}}\frac {d}{d-1}
\,d\theta+I_1\right).
\end{align}
Thanks to \eqref{eq11.09} and \eqref{eq11.06}, we have
\begin{equation*}
\psi'(\rho)=\frac {\omega_{d-1}(\alpha-2)(d+\alpha-2)}{\omega_d(d-1)}
\int_0^\pi\frac{\rho(\sin\theta)^{d}}
{A(\rho,\theta)^{4-\alpha}}\,d\theta.
\end{equation*}
Therefore, $\psi'(\rho)<0$ for $\rho\in [0,1)\cap (1,\infty)$ and, by the continuity of $\psi$, we conclude that $\psi$ is non-increasing on $[0,\infty)$. Finally, we consider the asymptotic
behavior of $\psi(\rho)$ as $\rho\to \infty$. It follows from \eqref{eq11.54} and \eqref{eq11.06} that
\begin{align*}
\psi(\rho)&=\frac {\omega_{d-1}}{\omega_d}\left(\int_0^\pi \frac {(\sin\theta)^{d-2}}{A(\rho,\theta)^{2-\alpha}}\,d\theta+\rho I_1\right)\\
&=\frac {\omega_{d-1}}{\omega_d}\int_0^\pi \frac {(\sin\theta)^{d-2}}{A(\rho,\theta)^{2-\alpha}}+\frac{\alpha-2}{d-1}
\frac {(\sin\theta)^{d}\rho^2}{A(\rho,\theta)^{4-\alpha}}\,d\theta.
\end{align*}
Therefore, $\psi(\rho)$ goes to zero as $\rho\to \infty$, since $\alpha<2$. Moreover,
\begin{align}
                                                \label{eq12.07}
\lim_{\rho\to\infty}\psi(\rho)\rho^{2-\alpha}
&=\frac {\omega_{d-1}}{\omega_d}\int_0^\pi (\sin\theta)^{d-2}+\frac{\alpha-2}{d-1}
(\sin\theta)^{d}\,d\theta\nonumber\\
&=\frac {\omega_{d-1}}{\omega_d}\int_0^\pi (\sin\theta)^{d-2}+\frac{\alpha-2}{d}
(\sin\theta)^{d-2}\,d\theta\nonumber\\
&=\frac{d+\alpha-2}{d}>0.
\end{align}
Here we used two elementary identities
\begin{align*}
&\int_0^\pi (\sin\theta)^d\,d\theta=\frac {d-1} d\int_0^\pi (\sin\theta)^{d-2}\,d\theta,\\
&\omega_{d-1}\int_0^\pi (\sin\theta)^{d-2}\,d\theta=\omega_{d}.
\end{align*}
Note that since $\psi$ is non-increasing, \eqref{eq12.07} also implies that $\psi$ is positive. The lemma is proved.
\end{proof}

The following lemma can be deduced from Lemma \ref{lem3.3} in the same way as Lemma \ref{lem2.4} is deduced from \eqref{lem2.3}. We omit the details and refer the reader to \cite[Sect. 4]{BLR}.

\begin{lemma}
                        \label{lem3.4}
Let $K(x)=|x|^\alpha,\alpha\in (0,2)$, and $(\mu_t)_{t\in [0,\infty)}$ be a radially symmetric weakly continuous measured-valued solution of the aggregation equation \eqref{Agg}. Then the vector field
$$
v(t,x)=-(\nabla K*\mu_t)1_{x\neq 0}
$$
is continuous on $[0,\infty)\times (\bR^d\setminus \{0\})$. For any $t\ge 0$, we have
$$
\mu_t=X_t^\#\mu_0,
$$
where, for each $x\in \bR^d$, $X_t=X_t(x)$ is an absolutely continuous function on $[0,\infty)$ satisfying
\begin{equation*}
\left\{
  \begin{aligned}
    \frac d {dt} X_t(x)&= v(t,X_t(x)) \quad\hbox{for a.e. $t\in (0,\infty)$;} \\
    X_0(x)&=x.
  \end{aligned}
\right.
\end{equation*}
Moreover, for any $x\neq 0$, $X_t(x)=R_t(|x|)x/|x|$, where $R_t$ is an absolutely continuous, non-negative and non-increasing function in $t\in [0,\infty)$, and
$$
\hat \mu_t=R_t^\#\hat \mu_0.
$$
Consequently, for any $x\neq 0$,
$$
(\nabla K*\mu_t)(x)=\alpha|x|^{\alpha-1}\int_0^\infty \psi(R_t(\rho)/|x|)\,d\hat \mu_0(\rho) \frac x {|x|},
$$
and $|(\nabla K*\mu_t)(x)|$ is non-decreasing in $t$.
\end{lemma}

Next we prove a point-wise lower bound of the velocity $v$ at $t=0$.

\begin{lemma}
                                \label{lem3.5}
Let $K(x)=|x|^\alpha,\alpha\in (0,2)$ and $u_0$ be the initial data defined in \eqref{eq22.16} with $\epsilon\in (0,1)$:
$$
u_0(x)=\frac L {|x|^{d+\alpha-2+\epsilon}}1_{|x|\le 1/2}.
$$
Then there exists a constant $\delta_1>0$ such that
\begin{equation}
                                    \label{eq22.22b}
|(\nabla K*u_0)(x)|\ge \delta_1 |x|^{1-\epsilon}
\end{equation}
for any $x\in \bR^d$ satisfying $0<|x|<1/2$.
\end{lemma}
\begin{proof}
Clearly, we have $\hat u_0(\rho)=L\omega_d \rho^{1-\alpha-\epsilon}1_{\rho<1/2}$. By the positivity and continuity  of $v(0,\cdot)$ in $\bR^d\setminus \{0\}$, it suffices to prove \eqref{eq22.22b} for $x\in \bR^d$ satisfying $0<|x|<1/4$. It follows from Lemma \ref{lem3.3} that for any $\rho\in (|x|,1/2)$,
\begin{equation}
                                \label{eq13.20}
\psi(\rho/|x|)\ge \delta_0 (|x|/\rho)^{2-\alpha}
\end{equation}
for a constant $\delta_0>0$ independent of $|x|$.
Thus, for any $x\in \bR^d$ satisfying $0<|x|<1/4$,
\begin{align*}
|(\nabla K*u_0)(x)|&\ge \alpha|x|^{\alpha-1}\int_{|x|}^{1/2}\psi(\rho/|x|)\hat u_0(\rho)\,d\rho\\
&\ge \alpha\delta_0L\omega_d|x|\int_{|x|}^{1/2}\rho^{-1-\epsilon} d\rho\\
&=\frac {\alpha\delta_0L\omega_d} {\epsilon}|x|\left(|x|^{-\epsilon}-(1/2)^{-\epsilon}\right)\\
&\ge \delta_1 |x|^{1-\epsilon},
\end{align*}
since $0<|x|<1/4$. The lower bound \eqref{eq22.22b} is proved.
\end{proof}

Now we are ready to prove Theorem \ref{thm2}.

\begin{proof}[Proof of Theorem \ref{thm2}]
With Lemma \ref{lem3.5}, the proof of Theorem \ref{thm2} is essentially the same as that of Theorem 4.2 \cite{BCL}. We present it here for completeness. By the definition of the push forward, for any $t>0$, we have
$$
\mu_t(\{0\})=(X_t^\#\mu_0)(\{0\})=\mu_0(X_t^{-1}(\{0\}).
$$
Note the solution of the ordinary differential equation $r'=-\delta_1 r^{1-\epsilon}$ with initial data $r_0>0$ reaches zero at a finite time $t=r_0^\epsilon/(\epsilon\delta_1)$. By Lemmas \ref{lem3.4} and \ref{lem3.5}, for any $t>0$, we can find $\delta>0$ such that $X_t(B_\delta)=\{0\}$. In other words, $B_\delta\subset X_t^{-1}(\{0\})$. Since $\mu_0(B_\delta)>0$, we conclude $\mu_t(\{0\})>0$ for any $t>0$. The theorem is proved.
\end{proof}

\subsection{Proof of Theorem \ref{thm3}}

The following lemma is an analogy of Lemma \ref{lem2.5}.

\begin{lemma}
                                \label{lem4.5}
Let $K(x)=|x|$ and $u_0$ be the initial data defined in \eqref{eq22.17} with $k\in (1/p_s,1)$:
$$
u_0(x)=\frac L {|x|^{d+\alpha-2}(-\log |x|)^{k}}1_{|x|\le 1/2}.
$$
Then there exists a constant $\delta_1>0$ such that
\begin{equation}
                                    \label{eq22.22c}
|(\nabla K*u_0)(x)|\ge \delta_1 |x|(-\log |x|)^{1-k}
\end{equation}
for any $x\in \bR^d$ satisfying $0<|x|<1/2$.
\end{lemma}
\begin{proof}
Clearly, we have
$$
\hat u_0(\rho)=L\omega_d\rho^{1-\alpha}(-\log \rho)^{-k}1_{\rho\le 1/2}.
$$
As before, it suffices to prove \eqref{eq22.22c} for $x\in \bR^d$ satisfying $0<|x|<1/4$. It follows from Lemma \ref{lem3.3} that, for any $\rho\in (|x|,1/2)$, \eqref{eq13.20} holds for a constant $\delta_0>0$ independent of $|x|$.
Thus, for any $x\in \bR^d$ satisfying $0<|x|<1/4$,
\begin{align*}
|(\nabla K*u_0)(x)|&\ge \alpha|x|^{\alpha-1}\int_{|x|}^{1/2}\psi(\rho/|x|)\hat u_0(\rho)\,d\rho\\
&\ge \alpha\delta_0L\omega_d|x|\int_{|x|}^{1/2}(-\log \rho)^{-k}\,\frac {d\rho}{\rho}\\
&=\alpha\frac {\delta_0L\omega_d} {1-k}|x|\left((-\log |x|)^{1-k}-(\log 2)^{1-k}\right)\\
&\ge \delta_1 |x|(-\log |x|)^{1-k},
\end{align*}
since $0<|x|<1/4$. The lemma is proved.
\end{proof}

\begin{proof}[Proof of Theorem \ref{thm3}]
We follow the proof of Theorem \ref{thm1} with some modifications. Thanks to Lemmas \ref{lem4.5} and \ref{lem3.4}, \eqref{eq23.27} remains true for any $r\in (0,1/2)$. Recall that $p_s=d/(d+\alpha-2)$. Now we suppose that, for some $t>0$, $\mu_t$ has a density function $u(t,x)\in L_{p_s}(\bR^d)$. By H\"older's inequality, for any $r\in (0,1/2)$
\begin{align}
\mu_t(B_{R_t(r)})&=\int_{B_{R_t(r)}}u(t,x)\,dx\nonumber\\
&\le \left(\int_{B_{R_t(r)}}u^{p_s}(t,x)\,dx\right)^{\frac {1} {p_s} }|B_{R_t(r)}|^{1-\frac {1} {p_s}  }\nonumber\\
&\le \|u(t,\cdot)\|_{L_{p_s}(\bR^d)}(R_t(r))^{2-\alpha}.
                                    \label{eq23.42c}
\end{align}
On the other hand, by the definitions of $\mu_t$ and $R_t$,
\begin{equation}
                                \label{eq23.49c}
\mu_t(B_{R_t(r)})\ge \mu_0(B_r)=L\omega_d\int_0^r \rho^{1-\alpha}(-\log\rho)^{-k}\,d\rho.
\end{equation}
We combine \eqref{eq23.42c}, \eqref{eq23.49c} and \eqref{eq23.27} to get
\begin{align}
\|u(t,\cdot)\|_{L_{p_s}(\bR^d)}&\ge \frac {\mu_t(B_{R_t(r)})} {(R_t(r))^{2-\alpha}}\nonumber\\
&\ge L\omega_d\frac {\int_0^r \rho^{1-\alpha}(-\log\rho)^{-k}\,d\rho}{r^{2-\alpha} e^{-(2-\alpha)\delta_2 t (-\log r)^{1-k}}}.
                                    \label{eq00.57c}
\end{align}
However, by L'Hospital's rule,
\begin{align*}
&\lim_{r\searrow 0}\frac {\int_0^r \rho^{1-\alpha}(-\log\rho)^{-k}\,d\rho}{r^{2-\alpha} e^{-(2-\alpha)\delta_2 t (-\log r)^{1-k}}}\\
&\,=\lim_{r\searrow 0}\frac {r^{1-\alpha}(-\log r)^{-k}}
{e^{-(2-\alpha)\delta_2 t (-\log r)^{1-k}}r^{1-\alpha}(2-\alpha)\big(1+\delta_2 t(1-k)(-\log r)^{-k}\big)}\\
&\,=\infty.
\end{align*}
This gives a contradiction to \eqref{eq00.57c} since we assume $u(t,\cdot)\in L_{p_s}(\bR^d)$. The theorem is proved.
\end{proof}



\end{document}